\documentclass[12pt]{amsart}
\usepackage{latexsym,fancyhdr,amssymb,color,amsmath,amsthm,graphicx,listings,comment}
\usepackage[section]{placeins}  \newtheorem{thm}{Theorem} 
\newtheorem{coro}{Corollary}  \let\paragraph\subsection
\title{Discrete Algebraic sets in Discrete Manifolds}
\author{Oliver Knill}
\date{12/21/2023}
\address{Department of Mathematics \\ Harvard University \\ Cambridge, MA, 02138 }
\begin{document}
\maketitle

\begin{abstract}
A discrete $d$-manifold is a finite simple graph $G=(V,E)$ where all unit spheres are $(d-1)$-spheres. 
A $d$-sphere is a $d$-manifold for which one can remove a vertex to make it contractible. A graph
is contractible if one can remove a vertex with contractible unit sphere to get a contractible graph.
We prove a discrete Morse-Sard theorem: if $G=(V,E)$ is a $d$-manifold and 
$f:V  \to \mathbb{R}^k$ an arbitrary map, then for any $c \notin f(V)$,
a level set $\{ f = c \}$ is always a $(d-k)$-manifold or empty. 
While a priori open sets in the simplicial complex of $G$, they are sub-manifolds
in the Barycentric refinement of $G$. Level sets are orientable if $G$ is orientable. 
Any complex-valued function $\psi$ on a discrete $4$-manifold $M$ 
defines so level surfaces $\{\psi=c\}$ which are except for $c \in f(V)$ always 
$2$-manifolds or empty. 
\end{abstract} 

\section{Introduction}

\paragraph{}
We continue to work on the Morse-Sard theme \cite{Morse1939,Sard1942} 
in a finite graph theoretical setting. Graphs come with a simplicial complex in which the
set of vertex sets $G$ of complete sub-graphs are the elements.
An Alexandrov topology on $G$ \cite{Alexandroff1937} is obtained by taking
the simplicial sub-complexes as closed sets \cite{FiniteTopology}. 
While \cite{KnillSard} looked at one hyper-surface at a time, 
needing the $k$'th Barycentric refinement to define $H=\{ f_1=0,\dots, f_k =0 \}$,
we define here arbitrary discrete co-dimension $k$ varieties in
a discrete $d$- manifold: for all $d$-manifolds $G$ and all $f:G \to \mathbb{R}^k$,
and all $c \notin f(G)$, the discrete algebraic set $H$ is either a $(d-k)$ 
manifold or empty. 
%A local non-degeneracy condition replaces the maximal rank condition 
%for the Jacobian matrix $df$ in differential topology.
%Discrete sphere geometry replaces Grassmannian geometry in the continuum. 

\begin{figure}[!htpb]
\scalebox{1.0}{\includegraphics{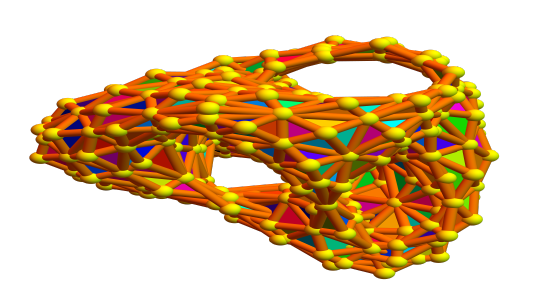}}
\label{rp4}
\caption{
We see a random level surface of a function $f: G=\mathbb{R}\mathbb{P}^4 \to \mathbb{C}$.
The host complex $\mathbb{R}\mathbb{P}^4$ has the f-vector $(16, 120, 330, 375, 150)$,
the Betti vector $(1,0,0,0,0)$ and the Euler characteristic $1$. 
Unlike in an orientable host manifold, the surface $H=\{ f = c\}$ can be 
non-orientable. In this case, we have hit a Klein bottle of genus $2$. It has the f-vector
$(76,174,94)$, Betti vector $(1,5,0)$ and Euler characteristic $-4$. 
}
\end{figure}

\paragraph{}
Our new discrete Morse-Sard statement does not need to impose any conditions on the functions $f_i$.
The theorem holds for the vector space of all functions $V \to \mathbb{R}^k$. 
In the discrete, we have explicit control about the set of regular values. 
In comparison, for manifolds in the continuum, the classical Morse-Sard theorem only implies that
it is a set of zero measure. Consider for example a spin system $f: G \to \{-1,1\}$ on a 
discrete $m$-manifold. The discrete Morse-Sard theorem tells that $\{ f =0 \}$ is either a 
$(m-1)$-manifold or empty. When taking several functions, we can so generate small
models of higher-dimensional manifolds and so get probability spaces of such structures. 
Level sets are given at first as delta sets and are in general much smaller
than their simplicial complex implementations. 
We can visualize the Barycentric refinements as simplicial complexes of graphs. 

\paragraph{}
Among discrete geometry settings, the language of {\bf finite simple graphs} 
is accessible and is a built-in data structure in many computer algebra implementations. The code
section below in this article shows some computer algebra implementation of which illustrates this. 
Also more general structures like simplicial complexes or {\bf delta sets} can be visualized as graphs.
A {\bf finite abstract simplicial complex} for example is a finite set of non-empty sets, closed
under the operation of taking non-empty subsets. \cite{DehnHeegaard}. Taking these sets as nodes and 
connecting two if one is contained in the other produces a graph with the same topological
features than the complex. It can be useful to extend the frame work slightly and work on 
{\bf delta sets}, a structure that still has an elementary singular chain complex and is more 
robust with respect to essential operations: 
{\bf 1)} taking open sets in simplicial complexes or 
{\bf 2)} taking quotients to see the host as a branched cover 
{\bf 3)} taking products of simplicial complexes and 
{\bf 4) }taking level set $S=\{ f=0\}$ of maps. 
For the later, we take the sets on which the functions change sign and impose some
non-degeneracy condition, replacing the notion of the Jacobean $df$ having maximal rank. 
While the level sets in a simplicial complex is an open set at first, 
their graph realizations define again closed sets. 

\paragraph{}
The face maps of a delta set allow to compute the cohomology groups directly using
kernels of concrete Hodge Laplacian matrix defined without a Barycentric refinement. 
This is important because the computation of the delta set representing a $d$-manifold 
takes a factor of at least $(d+1)!$ less time than computing it for the 
simplicial complex of the graph. For cohomology, 
the time improvement is even more dramatic, as we would need compute with much larger matrices, leading to 
a change of complexity that is polynomial in $(d+1)!$. We still want to realize the level set $S$
as a graph because we want to ``see" it as a d-manifold and we have the above natural definition
of manifold only for graphs or simplicial complexes. Also the delta set version of the Cartesian product
of two manifolds is a manifold, once realized as a graph. 
The cohomology of the product of two simplicial complexes can be computed faster
without realizing it as a graph. In \cite{KnillKuenneth}, we had used a product which is 
non-associative and related to the Stanley-Reisner ring. 
It was the product of delta sets followed by a graph realization. 

\paragraph{}
An example of a symmetry for manifolds are Dehn-Sommerville relations. They are very general 
\cite{Sommerville1927,Klee63,DehnSommerville,dehnsommervillegaussbonnet}. 
Here is an example: the f-vector of any $4$-manifold is perpendicular to the {\bf Dehn-Sommerville vector} 
$(0, -22, 33, -40, 45)$. The later is an eigenvector of the transpose of the Barycentric refinement matrix.
Dehn Sommerville helped us to see that an odd dimensional manifold has zero curvature. 
Let us look at an example: if $G$ is a 2-sphere with f-vector $(6,12,8)$, the 
f-vector of the delta set and 4-manifold
$G \times G$ is the convolution $(6,12,8)*(6,12,8)$, which is $(36, 144, 240, 192, 64)$. This is not 
perpendicular to the above Dehn-Sommerville vector. The f-vector of the graph representing the 4-manifold 
$S^2 \times S^2$ however is $(676, 8928, 28992, 34560, 13824)$ which is perpendicular.
To compute the cohomology, the Hodge matrix of the delta set implementation of $S^2 \times S^2$ 
is a $676 \times 676$ matrix, while the Hodge matrix belonging to the simplicial complex of
the graph is a $86980 \times 86980$ matrix. This is a significant reduction of complexity. By 
starting with delta set representations, the size can be even much smaller. 

\paragraph{}
An other place where co-dimension 2 surfaces matter is when rewriting {\bf curvature} $K(x)$
of an even-dimensional manifold in terms of expectations of the
Euler characteristics of discrete algebraic set in $S(x)$
\cite{indexformula,eveneuler}. Thanks to Sard, it can conveniently be related to the Euler 
characteristic of a co-dimension $2$ manifold. For odd-dimensional manifolds $M$, this 
immediately gives that the curvature is constant zero, something which 
bothered us still in 2011 \cite{cherngaussbonnet} where we only had experimental
evidence yet. The Dehn-Sommerville symmetries allowed to see this. An other path to verify this was
through index expectation. The symmetric Poincar\'e-Hopf index $j_f(x) = i_f(x) + i_{-f}(x)$ 
can be written as $-\chi(M_f(x))$, where $M_f(x)$ is a level surface contained
in the unit sphere of $x$. The symmetric Poincar\'e-Hopf index are therefore constant
zero for odd-dimensional manifolds and curvature $K(x) = {\rm E}[j_f(x)]$ given as expectation
is then zero too. 

\paragraph{}
The co-dimension two picture generalizes what we know classically for 
Riemannian manifolds. Let us explain this in dimension $2$, where it is elementary
in the discrete and goes back to Eberhard \cite{Eberhard1891} in the discrete. It
took longer in the continuum (see \cite{Chern1979} for history). 
Let $M$ is a classical 2-dimensional Riemannian manifold, meaning now a continuum 
manifold, not a discrete manifold.
The index $j_f(x) = 1-\{ y \in S_r(x), f(y)=f(x)\}/2$ at a point $x \in M$ is for small enough 
positive $r$ and for a generic smooth $f$ is a {\bf divisor}. This means that it is an 
{\bf integer-valued function} on $M$, 
which is non-zero only at finitely many points. These points are located on the 
singularity set of $f$ on $M$. For {\bf Morse functions} $f$ on $M$, one has
$j_f(x)=-1$ at saddle points and $j_f(x)=1$ at maxima and minima. Poincar\'e-Hopf tells
that the number of maxima and minima minus the number of saddles is the Euler characteristic. 
When averaging $j_f(x)$ over all Morse functions one gets the traditional Gaussian curvature
provided the probability measure on Morse functions is locally homogeneous. This can be 
achieved by placing $M$ into an ambient Euclidean space and taking a rotationally symmetric
Haar measure on the probability space of all linear functions.

\begin{figure}[!htpb]
\scalebox{0.2}{\includegraphics{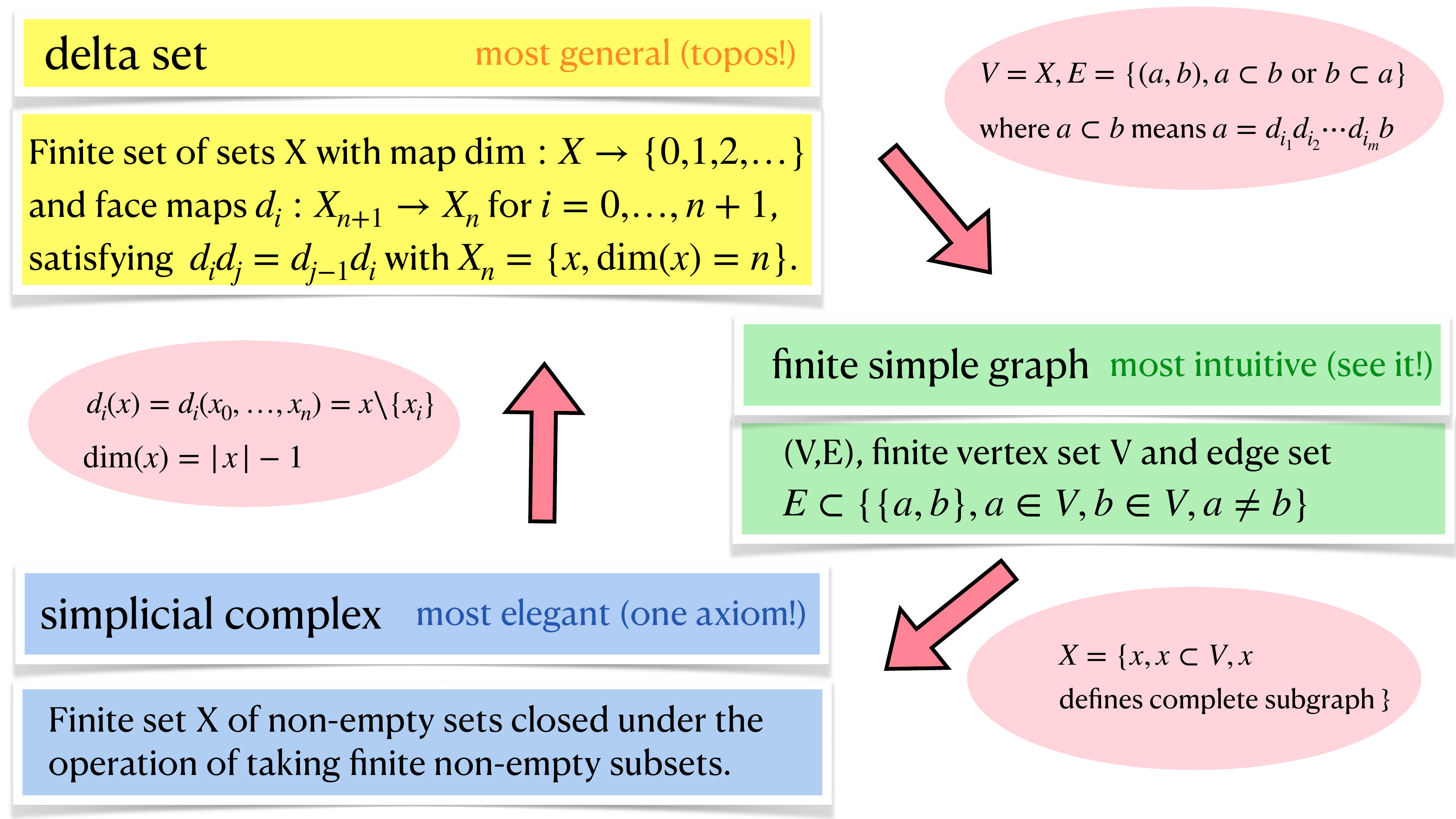}}
\label{trinity}
\caption{
Delta sets, simplicial complexes or graphs form a {\bf trinity
of structures}, where each structure has advantages. We see it
aligned with the Unix philosophy ``Simplicity, Clarity and Generality"
\cite{KernighanPike}. Going around once in this circle produces a 
Barycentric refinement. 
}
\end{figure}

\paragraph{}
When Nash-embedding a general compact Riemannian manifold into an Euclidean space $E$ 
\cite{EssentialNash} and taking
the expectation $K(x)={\rm E}[i_f(x)]$ over the probability space of all linear functions with
measure invariant under all rotations in the Euclidean space $E$ we get the normalized 
{\bf Gauss curvature} satisfying $\int_M K(x) \; dV$, if $dV$ is the normalized Riemannian volume measure.
The same works for any even-dimensional manifold and the index expectation is the 
Euler integrand in the {\bf Gauss-Bonnet-Chern} result for even-dimensional 
Riemannian manifolds. For discrete Gauss-Bonnet, see \cite{cherngaussbonnet,dehnsommervillegaussbonnet}.

\section{Discrete Morse Sard}

\paragraph{}
We start with giving the definition of a manifold. We have used these notions for many years now. They 
were motivated by notions in "digital 
topology" (in particular \cite{Evako1994,I94a}) and "discrete Morse theory" (which takes a slightly 
different approach like defining spheres through critical points \cite{Forman1999}). The history
of manifolds in the discrete and continuum is interest \cite{lakatos,Scholz,Dieudonne1989,Richeson}.
We start with the important notion of {\bf unit sphere} $S(v)$ of a vertex $v$ in a graph $G=(V,E)$. 
The graph $S(v)$ is the sub-graph generated by the set $W$ of vertices directly attached to 
$v$, meaning that $S(x)=(W,E_W)$, where $E_W \subset E$ consists of all edges $(a,b)$ in $E$
for which both $a,b$ are in $W$. The name "unit sphere" is justified because it is the unit sphere
in the metric space obtained by taking the geodesic distance as metric on the graph. 

\paragraph{}
The inductive definitions are close to piecewise linear (PL) geometry. It is however 
constructive in the sense that we can find out in finitely many computation steps whether a graph
is a manifold or not: 
a finite simple graph $G=(V,E)$ is a {\bf $d$-manifold} if every unit sphere
$S(v), v \in V$ is a $(d-1)$-sphere. A $d$-manifold is called a {\bf $d$-sphere} 
if the graph $G-v$ is contractible for some $v \in V$. A graph is {\bf contractible}
if both $S(v)$ and $G-v$ are contractible for some $v \in V$. 
The empty graph $0$ is the $(-1)$-sphere, the $1$-point graph $1$ is contractible. 

\paragraph{}
The reason to stick with ``contractible" rather than ``homotopic to 1" is because 
recognizing spheres is decidable as defined, while recognizing spheres using 
``1-homotopy" is not \cite{Novikov1955,Chernavsky,SphereRecognition}.
While the set-up does not matter for mathematical theorems and 
most results in this area can easily be bent to the larger non-constructive version of manifolds,
it matters when working with concrete models as a computer scientist,
where we consider it important that we can decide whether we deal with a manifold or not. 

\paragraph{}
Most theorems proven for the just defined constructive class obviously also hold for
the more general case. The definitions essentially covers {\bf PL geometry},
a category of geometries known to be much different from the category of {\bf topological manifolds}.
For example, there are complexes like the double suspension of a homology sphere which has 
a geometric realization that is homeomorphic to a sphere (by the double suspension theorem)
but which is not a sphere in the above sense because in the above definition, unit spheres
in unit spheres are spheres, while for a double suspension of a homology sphere, there are
2 vertices $v,w$, for which the unit sphere $S_{S(v)}(w)=S(w) \cap S(v)$ in $S(v)$ 
is a homology sphere. 

\paragraph{}
The following is a discrete version of the classical Morse-Sard theorem \cite{Morse1939,Sard1942}.
The classical theorem assured that for any differentiable function $f$ on a manifold, 
almost all $c$ are regular values in the sense that $f^{-1}(c)$ is a $(d-1)$-manifold (see \cite{Mil65}). 
In the discrete, the exception set is now known to be contained in $f(V)$. 
The set $\{ f=0 \}$ is defined as the set of all simplices in $G$ on which $f$ changes sign. 
It could be formulated differently: 
a simplicial map from a simplicial complex $G$ representing
a d-manifold to $K_2=\{ \{1\},\{2\},\{1,2\} \} $ has the property that $f^{-1}( \{ 1,2\} )$
represents a $(d-1)$-manifold. 

\paragraph{}
If $G$ is the simplicial complex of the geometry,
the level set is at first just an open set $U \subset G$ in the Alexandrov topology. When saying $U=\{ f=0 \}$,
we think of it as a {\bf graph} in which the sets of $U$ are the vertices and where two 
are connected if one is contained in the other. 
We let this graph represent what we mean with $\{f=0\}$. The level set $\{f=0\}$ 
is the same when we replace $f$ with ${\rm sign}(f)$, a function taking values in $\{-1,1\}$. 

\begin{thm}[Morse Sard for one function] 
If $G$ is {\bf $d$-manifold} and $f: V \to R$ is a real function,
then the level surface $\{ f = c \}$ for $c \notin f(V)$ is either
a $(d-1)$-manifold or empty. 
\end{thm} 

\begin{figure}[!htpb]
\scalebox{0.6}{\includegraphics{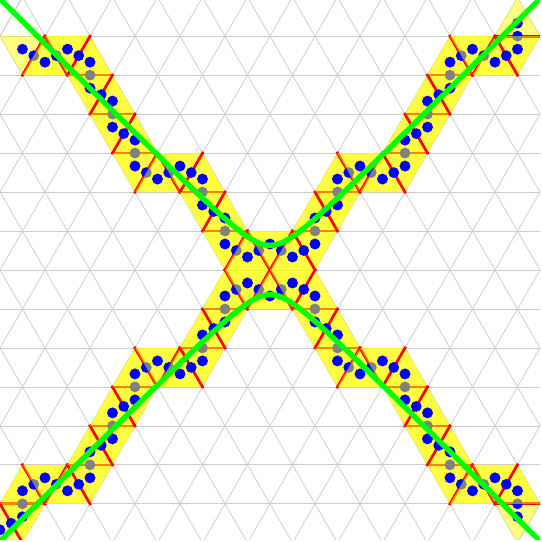}}
\scalebox{0.6}{\includegraphics{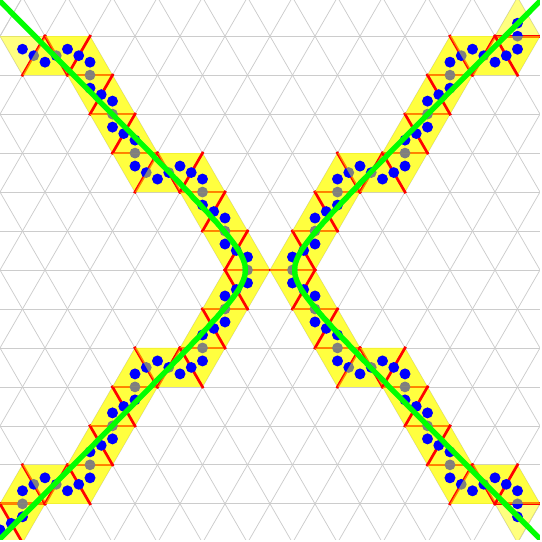}}
\label{potential}
\caption{
In this figure, we see part of a flat discrete $2$-torus. It is embedded in the plane
to conveniently attach $x,y$ coordinates to each vertex
and $f(x,y)=x^2-y^2$. We see the contours $\{ f=0.3\}$ and $\{f=-0.3\}$. 
These level sets $\{f=c\}$ consist of all edges and triangles on which $f-c$ changes sign. 
It is an open set in a simplicial complex but becomes a graph with 
simplices as vertices and where two are connected, if one is contained
in the other. 
}
\end{figure}

\paragraph{}
The graph $K=\{ f = c\}$ is a sub-graph of the
{\bf Barycentric refinement graph} $G_1=(V_1,E_1)$ in which the complete sub-graphs $x$ 
of $G$ are the vertices and where two such vertices are connected if one is contained in the other. 
The graph $K$ has as vertices all complete sub-graphs of $G$ on which $f-c$ changes sign. 
While the classical Morse-Sard lemma \cite{Morse1939,Sard1942} tells that the critical values have 
measure zero, the discrete case more explicitly states that {\bf all} values different 
from $f(V)$ are regular. As in the continuum case, also values not reached are considered regular
in the sense that they are not considered singular values. 

\paragraph{}
The proof of the discrete Morse-Sard theorem used that the {\bf Barycentric
refinement} $G_1$ of a graph $G$ has a natural {\bf hyperbolic structure} for the 
{\bf dimension function} ${\rm dim}(x)$. 
This function ${\rm dim}$ is naturally a discrete Morse function in the sense that 
both $S^-_f(x)$ and $S^+_f(x)$ are spheres at critical points.
For the dimension function, every $x$ is a critical point 
and the dimension $m={\rm dim}(x)$ is the {\bf Morse index} and 
$(-1)^{{\rm dim}(x)} = 1 -\chi(S^-_f(x))$ is the Poincar\'e-Hopf index. 
By definition, the unit sphere $S(x)$ of a vertex $x$ in $G_1$ consists of all simplices $S^-(x)$ 
which are either strictly contained in $x$ or the set of simplices $S^+(x)$ which strictly contain $x$. 
The simplicial complex $S^-(x)$ is a $(m-1)$-sphere in the sense that its
Barycentric refinement is a $(m-1)$-sphere graph. The set $S^+(x)$, also called {\bf star} of $x$,
is an open set, but it is not a simplicial complex in general. 
It is a $(d-k-1)$-sphere if we look at 
its Barycentric refinement, where the elements are the vertices and a connection holds if one 
is contained in the other. The Barycentric refinement graphs of $S^-(x)$
and $S^+(x)$ are both spheres if $G$ is a manifold and the unit sphere $S(x)$ is 
equal to $S^-(x) \oplus S^+(x)$, where $\oplus$ is the {\bf Zykov join} \cite{Zykov}.
The join operation is a monoid structure on graphs and
spheres are a sub-monoid. The $(-1)$-sphere $0$ is the zero element. 
To see that $K = \{ f = c \}$ of $G_1$ is a $(d-1)$-manifold
note that because $f$ changes sign on $x$ it also changes sign on any $y$ which 
contains $x$ so that $S_G^+(x)=S_K^+(x)$.
The open set $S_K^-(x)$ is a hyper-sphere in $S_G^-(x)$. It is obtained from $S_G^-(x)$ 
So, $S_K(x) = S^-_K(x) \oplus S^+_K(x)$ is a $(d-2)$-sphere and $K$ is a $(d-1)$-manifold. 

\section{Two functions}

\paragraph{}
So far, we covered essentially what was done already in \cite{KnillSard}. Here is a new
approach to the case when we have 2 or more functions. 
If $f,g$ are two functions on $G$, then the set of simplices on which both $f,g$ change sign is 
not a manifold in general. It is much too ``fat" (indeed, we will see it soon as a union
of manifolds). In 2015, we opted to extend the function 
$g$ to the Barycentric refinement $G_1$ of $G$ and to look at the level surface 
$\{ g = d \}$ in the sub-graph $\{ f = c \}$ of $G_1$. While it works, it
was unsatisfactory for two reasons, explained in the next paragraph. 

\paragraph{}
First of all, $\{ f = c, g = d\}$ was not the same than $\{ g=d, f=c \}$. 
Second, we needed to extend $g$ to the Barycentric refinement so that
the graph $\{ f = c, g = d\}$ is a sub-graph of the second Barycentric refinement of $G$. 
The number of vertices grows exponentially with every refinement. 
The $f$-vector $(f_0,f_1, \dots, f_d)$ counting the number of complete sub-graphs 
gets mapped into $A^2 f$, with the Barycentric operator $A_{ij} = S_{ji} i!$ involving 
{\bf Stirling numbers of the second kind}. If $G$ is the smallest possible 4-manifold with 
$f=(10, 40, 80, 80, 32)$, then the second Barycentric refinement $G_2$ already has the f-vector 
$A^2f = (24482, 303840, 970560, 1152000, 460800)$.
When experimenting with co-dimension 2 level surfaces in a 4-manifold, we would need to
deal with graphs of at least 24 thousand nodes. 
The first Barycentric refinement $G_1$ has a reasonable size because 
$Af =(242, 2640, 8160, 9600, 3840)$. Morse-Sard dealing with several functions at once 
therefore is a relief in the sense that it reduces complexity. 

\paragraph{}
Let us first consider the case of two functions $f,g$ on a $4$ manifold $G$. 
A motivation for the case of two functions is {\bf quantum mechanics} 
because {\bf waves} in quantum mechanics are modeled by complex-valued functions. 
The usual notion for such functions is $\psi$. Considering $\psi(v) = f(v) + i g(v)$ is equivalent to 
have two real valued functions $f,g$. The reason for taking dimension $4$ for the host is 
that this is the smallest dimension, where level surfaces $\{ \psi=0 \}$ become interesting; 
in dimension $2$ we would just get points and in dimension $3$ just get unions of cyclic graphs. 
If $G$ is 4-dimensional real manifold, we can also think of it as a 2-dimensional
{\bf complex surface}. The complex curve, a 1-dimensional complex manifold, 
$\{ \psi=0 \}$ is according to the classical Morse-Sard a real manifold. 
In the finite, where we can look at any function as ``analytic",
we are allowed to see level surfaces as complex curves. 
But what do we mean with $\{ \psi =0 \}$?

\paragraph{}
{\bf Definition:} 
Define $\{ \psi = c \}$ as the sub-graph of the Barycentric refinement $G_1$ of $G$,
generated by the vertices which are simplices in $G$ which contain a triangle $x$ on which 
$\psi-c = f + i g$ has the property that \\
(i) ${\rm sign}(f+ig)$ takes at least 3 different values  on $x$ \\
(ii) $(f + i g)(x)$ must includes the values $e=\{ -1+i,1-i \}$. 
This generalizes the real function case, where we had defined $\{f = c\}$ as the set of 
simplices which contain an edge $x$ on which $f-c$ changes sign. 

\paragraph{}
There is a choice involved. We could for example have forced
that $e=\{ -1-i,1+i\}$ is included, a case which would have corresponded to the wave 
$\overline{\psi}$, the {\bf complex conjugate}. 
In the continuum, $\{ \psi=0 \}$ and $\{ \overline{\psi}=0 \}$
are the same. In the discrete, these are most of the time different surfaces with different
features. Only if $\psi$ is sufficiently smooth, the two surfaces agree topologically. 
Indeed, for small $G$, the difference between $\{ \psi=0 \}$ and $\{ \overline{\psi}=0 \}$
is significant. It raises many interesting questions. The two cases were not the only ones.
With $k$ functions, there are $2^{k-1}(2^k-1)$ cases which is $6$ for $k=2$. 
Every of the $6$ edges $e$ in the complete graph $U$ with 4 vertices 
$\{ \pm 1, \pm 1\}$ produces a different surface $\{ \psi = 0 \}_e$ in general. 

\begin{thm}[Morse-Sard for complex functions on a d-manifold]
If $G=(V,E)$ is a $d$-manifold and $\psi: V \to \mathbb{C}$ is a
complex-valued function. Then, for all $c \notin \psi(V)$, the level
surface $K=\{ \psi =c \}$ is a $(d-2)$-manifold or empty. 
\end{thm}

\begin{proof}
The proof is very similar to the case with one function. 
If $K=\{ \psi=0 \}$ is not empty, we need to show that every unit sphere $S_K(x) \subset S_G(x)$
of $x$ is a $(d-3)$ sphere. Define $m={\rm dim}(x)$. To see that $S_K(x)$ is a $d-3$
sphere, use the decomposition 
$S_G(x) = S^-_G(x) \oplus S^+_G(x)$, where $S^-(x)$ is a $m-1$ sphere and $S^+(G)$ is a $d-m-1$ sphere. 
We still have $S^+_K(x) = S^+_G(x)$ and that $S^-_K(x)$ is a co-dimension
$2$-manifold in the $S^-_G(x)$. Because $S^-(G)$ is the boundary sphere of a simplex for 
which every sub-manifold is a sphere, $S^-(K)$ is a $(m-3)$-dimensional sphere.
Now $S_K(x)  = S_K^-(x) \oplus S_K^+(x)$ is the join of $(k-3)$-sphere 
and $(d-k-1)$-sphere which produces a $(d-3)$-sphere.
\end{proof} 

\begin{figure}[!htpb]
\scalebox{0.6}{\includegraphics{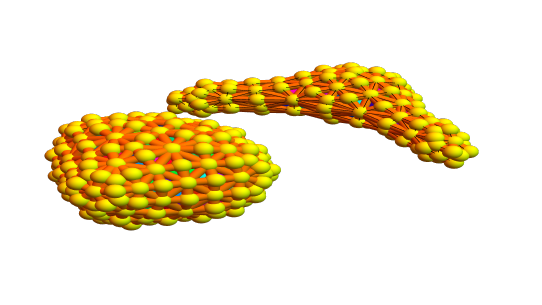}}
\scalebox{0.6}{\includegraphics{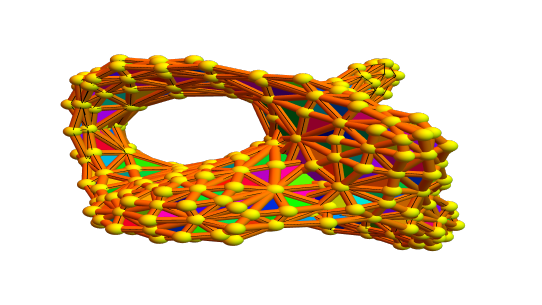}}
\label{conjugate}
\caption{
For a fixed $\psi$, we display the level surfaces $\{ \psi=0 \}$ and 
$\overline{\psi}=0\}$ in a small $\mathbb{RP}^4$. The complex $K_4$ (for
the 4 possible sign values) 6 edges leading to 6 different type of surfaces. 
They can differ topologically only in the small. If $G$ is larger and $\psi$ 
is smooth enough, then they have the same topology. 
}
\end{figure}

\paragraph{}
Here is a more general set-up for co-dimension-2 manifolds: look at the 3-dimensional simplex 
$U=K_4$ with vertices $\{ (-1,-1),(-1,1),(1,-1),(1,1) \}$. Given one of the 6 edges $e$ in $U$, 
there are two triangles $t_1,t_2$ containing $e$. Define 
$$   \{ \psi = 0 \}_e = \{ x \in G, {\rm sign}(\psi(x))  \in \{t_1,t_2\} \},  $$
where ${\rm sign}(a+ib) = {\rm sign}(a) + i {\rm sign}(b)$. Each of these 6 level sets
$\{ \psi = 0 \}_e$ is a co-dimension 2-manifold. 

\begin{figure}[!htpb]
\scalebox{1}{\includegraphics{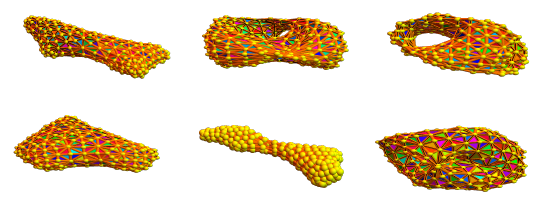}}
\label{flavor}
\caption{
Chose an edge $e$ in the 3-simplex $U$ defined by the corners $\{ \pm 1, \pm 1\}$. 
Define $\{ \psi=0\}_e = \{ x \in G, \psi(x)$ contains at least one 
triangle containing $e \}$. This produces in general $6$ different surfaces for $2$
functions and $(2^k-1) 2^{k-1}$ if we have $k$ functions. If $p=k$ and $2^p-1$ are
prime, then we have a "perfect number" of surfaces. 
}
\end{figure}

\paragraph{}
What happens if we take the inverse image of one of the 4 triangles $t$ in $U$
If $\psi: G \to \mathbb{C}$ does not take the value $0$, define 
$\Psi(x) = {\rm sign}(\psi(x))$. The following result allows us to produce manifolds with 
boundary. 

\begin{thm}
If $G$ is a $d$-manifold and $t$ is a triangle in $U$, then $\Psi^{-1}(t)$ 
is a $(d-2)$-manifold with boundary. 
\end{thm} 

\begin{proof}
If $t=t_1$ is a triangle in $U$ and $t_2$ is an other triangle in $U$ which shares
with $t_1$ an edge $e$. By the previous result 
$$ H_1 \cup H_2 = \Psi^{-1}(t_1) \cup \Psi^{-1}(t_2) $$
is a $(d-2)$ manifold $H$ without boundary. Obviously $H_1=\Psi^{-1}(t_1)$ and 
$H_2 = \Psi^{-1}(t_2)$ are both part of this manifold $H$. Their intersection 
$H_1 \cap H_2$ consist of all $x$ such that $\Psi(x)$ is both in $t_1$ and $t_2$. 
This means that $\Psi(x)$ must reach all of $U$. This must be a $(d-3)$ manifold
(next section) as it can be seen as the level surface $\{ f=0,g=0,h=0\}$ where $h$ 
is a constant function not taking the value $0$. 
Since $H_1 \cap H_2$ is a $d-3$ manifold and $H_1 \cup H_2$ is a $d-2$ manifold,
both $H_1,H_2$ are $(d-2)$-manifolds with boundary. 
\end{proof} 

\begin{figure}[!htpb]
\scalebox{0.9}{\includegraphics{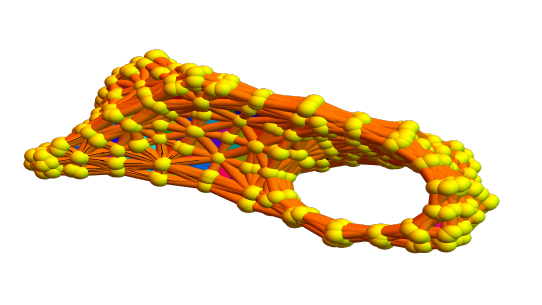}}
\label{solidtorus}
\caption{
As an application, we can produce solid 3 manifolds with 
boundary by taking a constructing a co-dimension 2 manifold 
$\psi^{-1}(t)$ in a 5-manifold for a given triangle. 
In this case we took for $G$ a suspension of $RP^4$ and 
got a solid 2-torus. 
}
\end{figure}

\paragraph{}
Here is a possible reformulation. We hope to generalize this much 
more in a future work.

\begin{coro}
If $G$ is a $d$-manifold and 
if $t_1,t_2$ are two triangles in $U$ intersecting in an edge $e$, then 
$\{ \psi = 0 \}_e$ is obtained by gluing $\Psi^{-1}(t_1)$ and 
$\Psi^{-1}(t_2)$ along a $d-3$ manifold. 
\end{coro}

\paragraph{}
This works already in dimension $d=3$ and a complex function $\psi$. 
The set $\psi^{-1}((1+i,i-1,-1-i))$ consists of a finite set of path graphs.
Also $\psi^{-1}((1+i,-1+i,-1-i))$ consists of a set of path graphs. 
The union is a choice of $\{ \psi=0 \}$  and is a 1-manifold, a union of
cyclic graphs obtained by gluing the two parts along $0$-dimensional manifold,
a set of isolated points. 

\begin{figure}[!htpb]
\scalebox{0.6}{\includegraphics{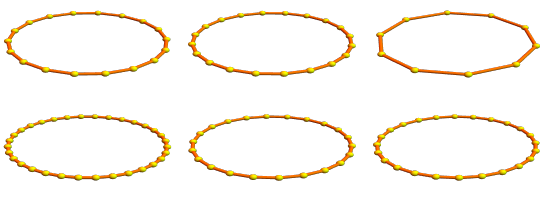}}
\scalebox{0.6}{\includegraphics{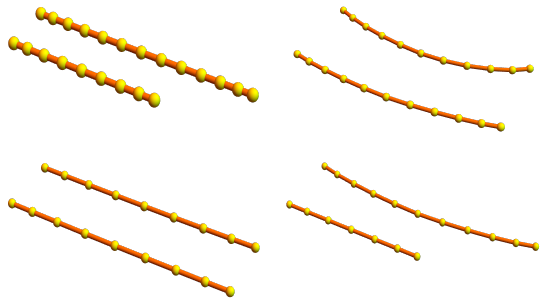}}
\label{flavor3d} 
\caption{
We first look at the level surfaces of a random complex-valued
function in a 3-dimensional manifold $G$. 
The second picture shows the manifolds $H$ with boundary 
obtained by just taking the inverse of a simplex.  }
\end{figure}

\paragraph{}
We can rephrase this that a continuous map from a $d$-manifold $G$ to $K_3$ has
$f^{-1}({1,2,3})$ as the empty set or an open set becomes in the Barycentric 
refinement a $(d-2)$-manifold with boundary.
One dimension lower, a continuous map from G to $K_2$ has
$f^{-1}(\{1,2\})$ either an empty set or a $(d-1)$-manifold with boundary.
For more on the Alexandrov topology for simplicial complexes, 
see \cite{KnillTopology2023}. This does not yet like being the most
general set-up even with 2 functions. We will try to extend this to general
simplicial maps from a $d$-manifold $M$ to a $3$-manifold with a hyper surface
$N$. The inverse $f^{-1}(N)$ appears then to be a $d-2$ manifold but
much larger in general than $M$. Every $y \in N$ has as an inverse image either
a $d-2$ manifold with boundary or a $d-3$ manifold gluing these $d-2$ patches
together to a $d-2$ manifold. 

\section{Arbitrary many functions}

\paragraph{}
{\bf Definition:} given $k$ functions $f_1,\dots, f_k$ in a $d$-manifold $G$, 
define $\{ f_1=0, \dots, f_k=0\}$ as the set of simplices $x$ in $G$ 
for which the image ${\rm sign}(f)(x)$ has at least $k+1$ elements and
such that it includes a fixed set of $k$ values like
for example $(-1,1,\dots, 1),(1,-1,\dots, 1),(1,1,\dots,-1)$. 

\paragraph{}
As in the case $k=2$, there is more than one surface. 
This is one of the ${\rm Binomial}(2^k,2)=2^{k-1}(2^k-1)$ 
different choices we can do. If $p=k$ is prime and $2^p-1$ are
prime (the later is then called a {\bf Mersenne prime}), 
then $n=2^{k-1}(2^k-1)$ is a {\bf perfect number} in the sense
of number theory: the sum of proper divisors of $n$ is equal to $n$. 
(One of the oldest problem in mathematics asks whether there are 
odd perfect numbers).  All these surfaces can differ for small 
$G$ or random $f$ but for large $G$ and smooth enough $f$,
all have the same topology. 

\begin{thm}
For any $d$-manifold $(V,E)$, and any set of functions $f_1,\dots, f_k$, 
and any $c \notin f(V)$, the level surface $\{ f = c\}$ is a $(d-k)$-manifold 
or empty. 
\end{thm}

\begin{proof}
The unit sphere $S_K(x) = S_K^-(x) \oplus S_K^+(x) = S_K^-(x) \oplus S_G^+(x)$
is a sphere if $S_K^-(x)$ is a sphere. But $S_K^-(x)$ is a level surface in 
the $(m-1)$-sphere $S_G^-(x)$. 
\end{proof} 

\paragraph{}
Morse-Sard boils down to a local statement about level surfaces in boundary spheres of simplices.
We only have to assure that the level surface $\{f = 0\}$ in such a sphere is a sphere. 
Let $x=(x_0,\dots,x_m)$ be a $m$-dimensional simplex and 
$f_1,\dots,f_k: x \to \{-1,1\}$. Assume $(f_1,\dots,f_k)$ takes at least $m+1$ values including
$(-1,1, \dots, 1), ( 1,-1, \dots, 1) , \dots (1,1, \dots, 1,-1)$. Then $\{ f =0 \}$ 
is a $(m-1-k)$-sphere.  If $k=0$, then $\{ f=0\}$ is the Barycentric refinement of the boundary 
sphere of $x$. If $k=1$, then $\{ f= 0 \}$ is the set of simplices on which $f$ changes sign.
This is always the join of two spheres, the sphere where $f$ is positive and the sphere, where
$f$ is negative, as long as there are two values.

\section{Integral geometry}

\paragraph{}
This section is a review of a result \cite{indexformula} formulated in 2012 
and \cite{eveneuler} in 2013, when
no discrete Sard was available yet. In \cite{poincarehopf} we gave the 
Poincar\'e-Hopf index $i_f(v) = 1-\chi(S^-_f(v))$ for a graph $G=(V,E)$ and coloring
$f$, where $S^-(v)$ is generated by the vertices $w$ of $S(v)$, where $f(w)<f(v)$.
Then $i_{-f}(v) = 1-\chi(S^+_f(v))$. The complement of the disjoint union 
$S^-(v) \cup S^+(v)$  in $S(v)$ is now the level set $\{ f = c \}$ with $c = f(w)$ 
in $S(v)$. Because the unit sphere $S(v)$ can be seen as a level set for 
the distance function to $v$, we actually deal with a contour surface 
of co-dimension $2$. If $G$ was a manifold, where $\chi(S(v))=(1+(-1)^{d-1}$, we got 
$j_f(v) = 1-\chi(\{ f=f(v) \}_{S_v})/2$ for even-dimensional manifolds and 
$j_f(v) = -\chi(\{ f=f(v) \}_{S_v})/2$ for odd-dimensional manifolds. 

\paragraph{}
The symmetric Poincar\'e-Hopf theorem $\chi(G) = \sum_{v \in V} j_f(v)$ 
therefore related the Euler characteristic $\chi(G)$ of an even-dimensional manifold $G$ 
equipped with a color function $f: G \to \mathbb{R}$ 
with the sum of the symmetric indices 
$j_f(v) = 1-\chi( \{ y \in S(v), f(y)=f(v) \})/2$. The {\bf curvature}
$K(v) = E_f[ j_f(v) ]$ was given as the expectation of
$1-E[ \{ f=f(x) \}]/2$ on the odd-dimensional unit sphere $S(x)$.
The probability space is a space of colorings on $G$:

\begin{thm}
For an even-dimensional manifold $K(x) = 1-{\rm E}[ \chi(\{ f=f(x) \})]/2$. \\
For an odd-dimensional manifold, $K(x) = -{\rm E}[ \chi(\{f = f(x) )\}]/2$. \\
\end{thm} 

\paragraph{}
This immediately implied: 

\begin{coro}
For an odd-dimensional manifold, the curvature is constant  $K(x)=0$. 
\end{coro}

\paragraph{}
We now can write Euler characteristic as $\chi(G) = 1-{\rm E}[X]/2$, where $X$
is the Euler characteristic of a ``random co-dimension $2$ surface. 
The probability space however is rather restricted. If we write it 
as a probability space of functions $\psi=f+ig$, we require
that for fixed $f$, we have $g(w)=1-1_v(w)$ and
that on the set $S(v) = \{ g=0 \}$ the function $f(v)$ is different
from all values on $S(v)$. It begs the question what happens
if we take the expectation of $\chi( \{ f+i g = 0 \} )$ when 
averaging over all possible functions $f+ig$ on the graph taking values 
in $\{ 1+i,1-i,-1+i,-1-i\}$. 

\section{Remarks}

\paragraph{}
The discrete Morse-Sard result illustrates
that moving from continuum geometries to finite quantum geometries becomes not only more 
regular but also more rich. While the translation of the objects from Euclidean manifolds 
to discrete manifolds is not obvious at first, the results are completely analog.
A point $x$ is a {\bf critical point} of a function $f$, if $M(x) = \{ f =f(x) \}$ 
in $S(x)$ is not a $(d-2)$-sphere. 
The symmetric Morse index $j(x) = ((1+(-1)^d)-\chi(M(x)))/2$ quantifies that. 
The sum $\sum_{x \in V_1} j(x)$ is the Euler characteristic of $G$. 
This is analog to the continuum. A smooth function $f$ has a critical
point $x$ on a smooth Riemannian d-manifold, if for small radius $r>0$, 
the set $M_r(x)=S_r(x) \cap \{f = f(x) \}$ is not a $(d-2)$-sphere. 
An example is a maximum or minimum of $f$, where $S(x) \cap \{ f = f(x)\}$ 
is empty. 

\paragraph{}
Given two functions $f,g$ on a $d$-manifold we can look at the {\bf discrete Lagrange problem}
where $f$ is given on the {\bf constraint} $\{ g=0 \}$. We would like to find
the extrema of $f$ on the level surface using Lagrange equations similarly as in the continuum. 
Critical points should be simplices in $g=0$, where the center manifold $\{ f= f(x) \}$ in $S(x)$ 
is not a $(d-3)$-sphere. For example, on a $2$-manifold, then $g=0$ is a $1$-manifold and $f$ 
can be extremized at points $y$ on this manifold $\{ g = 0 \}$, where $S(y)$ is
empty. This is what happens for a coloring $f$ on a $1$-dimensional manifold? A point 
$x$ could be called critical point if $\{ f=f(x) \}$ is not a zero sphere and
a regular point if $\{ f = f(x)$ is a $0$-sphere. 

\paragraph{}
Let us look at a triangle $x$ and two functions $f,g$
such that $f,g$ both change sign on $x$. Define the {\bf gradient}
of $f$ as the structure of the signs of $f$. The tail of the arrow
is the set where $f$ is negative. The head is the arrow where
$f$ is positive. For a simplex $x=(a,b,c)$ in the intersection $\{f =c,g =d\}$
the values $(f(a),f(b),f(c))$ and $(g(a),g(b),g(c))$ must have the properties
that the gradients have the same type but are not the same. The two possible
cases (up to rotation) are $(1,-1,-1),(-1,1,-1)$ and $(1,-1,1),(-1,1,1)$. 
If the sign structure is $(1,1,-1),(-1,-1,1)$ for example, we do not intersect
even so both functions change sign on $x$. 

\paragraph{}
Instead of starting with a graph, we could have begun with a finite abstract
simplicial complex $G$, a finite set of sets closed under the operation of taking finite 
non-empty subsets. It carries a topology
generated by the stars $U(x)=\{ y \in G, x \subset y\}$ which are the smallest 
open sets. The closed sets are then the sub-complexes. The boundary $S(x)$ of 
$U(x)$ is the unit sphere. As it is $S(x)=\overline{U}(x) \setminus U(x)$
it is closed as the complement of an open set in a closed set. A complex is a {\bf $d$-manifold}
if all unit spheres $S(x)$ are $(d-1)$-spheres. A $d$-manifold is a {\bf $d$-sphere} if there exists
$x$ such that the complex $G \setminus U(x)$ is contractible. A complex $G$ is contractible if
there exists $x$ such that both the complexes $S(x)$ and $G \setminus U(x)$ are contractible.  
These inductions start with $0=\{\}$ being the $(-1)$-sphere and $1=\{1\}$ being contractible. 

\paragraph{}
Here is an application of co-dimension $2$ manifolds. 
We can visualizations of eigenfunctions of Laplacians on 3-manifolds $\{ f =0 \}$ in 3-manifolds. 
Let $H$ be a $3$-manifold $H=\{f =0 \}$ in a 4-manifold $G$. We can now 
look at the Laplacian $L$ on $H$ and an eigenvector $f$ to an eigenvalue $\lambda$. 
Then $f$ is a real-valued function on $H$. For every original vertex $v$ of $G$ which
happens to be in at least one vertex $x$ of $H$ (so that $x$ is a simplex in $G$), 
then define the function $g(v)=\sum_{x, v \in x} \psi(v)$. 
Now, we can look at the manifold $\{ f=0, g=0  \}$ which is a nodal surface of the Laplacian. 
In other words, we can study {\bf Cladni surfaces} for Laplacians on surfaces. 

\pagebreak

\section{Code}

\paragraph{}
The following self-contained Mathematica code allows to compute level surfaces of real or complex
functions of arbitrary simplicial complexes and compute its cohomology (as an open set).
The closed manifold is obtained by adjusting the dimension. 
We encode a delta set as $(G,D,b)$, where $G$ is a set of sets $G$, $D$ is a Dirac matrix
and $b$ a dimension function. Two 3-manifolds $\mathbb{RP}^3$ and the homology sphere 
are from \cite{Manifoldpage}. 

\begin{tiny} \lstset{language=Mathematica} \lstset{frameround=fttt}
\begin{lstlisting}[frame=single]
Generate[A_]:=Sort[Delete[Union[Sort[Flatten[Map[Subsets,A],1]]],1]];
Whitney[s_]:=Generate[FindClique[s, Infinity, All]];  sig[x_]:=Signature[x]; L=Length; 
F[G_]:=Module[{l=Map[L,G]},If[G=={},{},Table[Sum[If[l[[j]]==k,1,0],{j,L[l]}],{k,Max[l]}]]]; 
sig[x_,y_]:=If[SubsetQ[x,y]&&(L[x]==L[y]+1),sig[Prepend[y,Complement[x,y][[1]]]]*sig[x],0];
nu[A_]:=If[A=={},0,L[NullSpace[A]]];
Dirac[G_]:=Module[{f=F[G],b,d,n=L[G]},b=Prepend[Table[Sum[f[[l]],{l,k}],{k,L[f]}],0];
 d=Table[sig[G[[i]],G[[j]]],{i,n},{j,n}]; {d+Transpose[d],b}];
Beltrami[G_]:= Module[{B=Dirac[G][[1]]},B.B];
Hodge[G_]:=Module[{Q,b,H},{Q,b}=Dirac[G];H=Q.Q;Table[Table[H[[b[[k]]+i,b[[k]]+j]],
 {i,b[[k+1]]-b[[k]]},{j,b[[k+1]]-b[[k]]}],{k,L[b]-1}]];
Betti[s_]:=Module[{G},If[GraphQ[s],G=Whitney[s],G=s];Map[nu,Hodge[G]]];
Fvector[s_]:=Module[{G},If[GraphQ[s],G=Whitney[s],G=s];Delete[BinCounts[Map[L,G]],1]];
ToGraph[G_]:=UndirectedGraph[n=L[G];Graph[Range[n],
  Select[Flatten[Table[k->l,{k,n},{l,k+1,n}],1],(SubsetQ[G[[#[[2]]]],G[[#[[1]]]]])&]]];
Barycentric[s_]:=If[GraphQ[s],ToGraph[Whitney[s]],Whitney[ToGraph[s]]];
Suspension[s_] :=Module[{v=VertexList[s],e=EdgeRules[s],m},m=Max[v];
 Do[e=Union[e,{v[[k]]->m+1,v[[k]]->m+2}],{k,L[v]}];UndirectedGraph[Graph[e]]];
USphere[s_,v_]:=VertexDelete[NeighborhoodGraph[s,v],v];   Rnd[x_]:=x -> RandomChoice[{-1,1}];
RFunction[s_]:=Module[{G},If[GraphQ[s],G=Whitney[s],G=s];Map[Rnd,Union[Flatten[G]]]];
LSurface[s_,f_]:=Module[{G,H={},K},If[GraphQ[s],K=Whitney[s],K=s];G=Select[K,L[#1]>1 &];
  Ch[b_]:=Module[{},A=Union[Table[Sign[b[[k]]],{k,L[b]}]];SubsetQ[A,{1,-1}]&&L[A]>=2];
  Do[If[Ch[G[[k]] /. f],H=Append[H,G[[k]]]],{k,L[G]}]; H];
LSurface[s_,{f_,g_}]:=Module[{G,H={},K},If[GraphQ[s],K=Whitney[s],K=s];G=Select[K,L[#1]>2 &];
  Ch[b_,d_]:=Module[{},A=Union[Table[Sign[{b[[k]],d[[k]]}],{k,L[b]}]];
     MemberQ[A,{1,-1}] && MemberQ[A,{-1,1}] && L[A]>=3]; 
  Do[If[Ch[G[[k]] /. f,G[[k]] /. g],H=Append[H,G[[k]]]],{k,L[G]}];H];

G = Whitney[Suspension[Suspension[PolyhedronData["Icosahedron", "Skeleton"]]]];
f=RFunction[G];g=RFunction[G]; gm=Table[g[[k,1]]->-g[[k,2]],{k,L[g]}];
H1=LSurface[G,{f,g}];H2=LSurface[G,{f,gm}]; 
Print[{Betti[H1],Betti[H2]}]; {GraphPlot3D[ToGraph[H1]],GraphPlot3D[ToGraph[H2]]}

RP3=Generate[{{1,2,3,4},{1,2,3,5},{1,2,4,6},{1,2,5,6},{1,3,4,7},{1,3,5,7},{1,4,6,7},
{1,5,6,7},{2,3,4,8},{2,3,5,9},{2,3,8,9},{2,4,6,10},{2,4,8,10},{2,5,6,11},{2,5,9,11},
{2,6,10,11},{2,7,8,9},{2,7,8,10},{2,7,9,11},{2,7,10,11}, {3,4,7,11},{3,4,8,11},
{3,5,7,10},{3,5,9,10},{3,6,8,9},{3,6,8,11},{3,6,9,10}, {3,6,10,11},{3,7,10,11},
{4,5,8,10},{4,5,8,11},{4,5,9,10},{4,5,9,11}, {4,6,7,9},{4,6,9,10},{4,7,9,11},{5,6,7,8},
{5,6,8,11},{5,7,8,10},{6,7,8,9}}];f=RFunction[RP3];H=LSurface[RP3,f];U=GraphPlot3D[ToGraph[H]] 

G=Generate[{{0,1,5,11},{0,1,5,17},{0,1,7,13},{0,1,7,19},{0,1,9,15},{0,1,9,21},{0,1,11,13},
{0,1,15,19},{0,1,17,21},{0,2,6,12},{0,2,6,18},{0,2,7,13},{0,2,7,19},{0,2,10,16},{0,2,10,22},
{0,2,12,13},{0,2,16,19},{0,2,18,22},{0,3,8,14},{0,3,8,20},{0,3,9,15},{0,3,9,21},{0,3,10,16},
{0,3,10,22},{0,3,14,22},{0,3,15,16},{0,3,20,21},{0,4,5,11},{0,4,5,17},{0,4,6,12},{0,4,6,18},
{0,4,8,14},{0,4,8,20},{0,4,11,12},{0,4,14,18},{0,4,17,20},{0,11,12,13},{0,14,18,22},
{0,15,16,19},{0,17,20,21},{1,2,5,11},{1,2,5,17},{1,2,8,14},{1,2,8,20},{1,2,10,16},{1,2,10,22},
{1,2,11,14},{1,2,16,20},{1,2,17,22},{1,3,6,12},{1,3,6,18},{1,3,7,13},{1,3,7,19},{1,3,8,14},
{1,3,8,20},{1,3,12,20},{1,3,13,14},{1,3,18,19},{1,4,6,12},{1,4,6,18},{1,4,9,15},{1,4,9,21},
{1,4,10,16},{1,4,10,22},{1,4,12,16},{1,4,15,18},{1,4,21,22},{1,11,13,14},{1,12,16,20},
{1,15,18,19},{1,17,21,22},{2,3,5,11},{2,3,5,17},{2,3,6,12},{2,3,6,18},{2,3,9,15},{2,3,9,21},
{2,3,11,15},{2,3,12,21},{2,3,17,18},{2,4,7,13},{2,4,7,19},{2,4,8,14},{2,4,8,20},{2,4,9,15},
{2,4,9,21},{2,4,13,21},{2,4,14,15},{2,4,19,20},{2,11,14,15},{2,12,13,21},{2,16,19,20},
{2,17,18,22},{3,4,5,11},{3,4,5,17},{3,4,7,13},{3,4,7,19},{3,4,10,16},{3,4,10,22},{3,4,11,16},
{3,4,13,22},{3,4,17,19},{3,11,15,16},{3,12,20,21},{3,13,14,22},{3,17,18,19},{4,11,12,16},
{4,13,21,22},{4,14,15,18},{4,17,19,20},{11,12,13,23},{11,12,16,23},{11,13,14,23},{11,14,15,23},
{11,15,16,23},{12,13,21,23},{12,16,20,23},{12,20,21,23},{13,14,22,23},{13,21,22,23},
{14,15,18,23},{14,18,22,23},{15,16,19,23},{15,18,19,23},{16,19,20,23},{17,18,19,23},
{17,18,22,23},{17,19,20,23},{17,20,21,23},{17,21,22,23}}];f=RFunction[G];H=LSurface[G,f];
Print[Betti[H]]; GraphPlot3D[ToGraph[G]] 
\end{lstlisting}
\end{tiny}

\bibliographystyle{plain}

\end{document}